\documentclass[11pt,reqno]{amsart}
\usepackage{amscd,amsmath,amsopn,amssymb,amsthm,multicol}
\usepackage{tikz,subdepth,anysize,verbatim,ifthen,xargs,colortbl,float}
\usepackage{longtable,mathtools}
\usepackage[english]{babel}
\usepackage[utf8]{inputenc}
\everymath=\expandafter{\the\everymath\displaystyle}
\usetikzlibrary{decorations.markings,arrows,arrows.meta,bending,calc}
\tikzset{>={Stealth[scale=1.5, bend]}}
\theoremstyle{plain}
\newtheorem{theorem}{Theorem}

\newtheorem{lemma}{Lemma}
\newtheorem{cor}{Corollary}
\newtheorem{rk}{Remark}
\newtheorem*{definition}{Definition}

\newcommand\com[1]{}

\newcommand\op[1]{\mathop{\rm #1}\nolimits}
\newcommand\p{\partial}

\newcommand\R{{\mathbb R}}

\arrayrulecolor{gray!50}

\begin{document}

\title{Rational First Integrals and Relative Killing Tensors}

\author{Boris Kruglikov}
\address{Department of Mathematics and Statistics, UiT the Arctic University of Norway, Troms\o\ 9037, Norway.
\ E-mail: {\tt boris.kruglikov@uit.no}. }

 \begin{abstract}
We relate rational integrals of the geodesic flow of a (pseudo-)Riemannian metric
to relative Killing tensors, describe the spaces they span and discuss upper bounds on their dimensions. 
 \end{abstract}

\maketitle

First integrals of motion of Hamiltonian systems help to integrate their dynamics. In the case of geodesic flows, 
homogeneous polynomial in momenta integrals, also known as Killing tensors, play an important role in applications;
they were extensively studied for more than a century. In particular, it was shown in \cite{D,Th} that for a metric $g$ 
on a manifold $M$ of dimension $m$ the (linear) space $\mathcal{K}_{m,d}(g)$ of such integrals of $\deg K=d$ 
is always finite dimensional with the sharp bound 
 $$
\dim\mathcal{K}_{m,d}(g)\leq\Lambda_{m,d}:=\frac{(m+d-1)!(m+d)!}{(m-1)!m!d!(d+1)!}
 $$
attained precisely on spaceforms (metrics of constant curvature; any metric signature and curvature sign).

Recently, rational in momenta integrals have attracted attention in Hamiltonian mechanics \cite{Ko}
and some explicit examples of those have been constructed \cite{AS} and investigated \cite{AA}. 
While they were introduced by Darboux \cite{D} and then studied in the context of general relativity \cite{Co,AHT},
much less is known about rational integrals than about their polynomial counterpart, due to nonlinearity of the
corresponding equation. 
We will make a few basic observations on rational integrals in this note.

\section{Rational integrals and relative Killing tensors}\label{S1}

Let $g$ be a Riemannian or pseudo-Riemannian metric on a manifold $M$ and $(g^{ij})=(g_{ij})^{-1}$
be the inverse tensor. The Hamiltonian $H(x,p)=\tfrac12g^{ij}(x)p_ip_j$ together with the canonical symplectic
structure $\omega$ on $T^*M$ define the Hamiltonian vector field $X_H=\omega^{-1}dH$. A function $F=F(x,p)$ is an integral 
if it is in involution with $H$ wrt the corresponding Poisson structure $(\pi^{ab})=(\omega_{ab})^{-1}$: $\{H,F\}=X_H(F)=0$.

Integrals of the form $F(x,p)=K^{i_1\dots i_d}(x)p_{i_1}\cdots p_{i_d}$ (summation by repeated indices assumed)
are Killing $d$-tensors, and they can be identified with smooth
sections of the bundle $S^dTM$, $K^{i_1\dots i_d}=K^{(i_1\dots i_d)}$, that are in the kernel of the operator
$\op{sym}\circ\nabla_g$: using $g$-lowering of indices we can write this as $K_{(i_1\dots i_d;i_{d+1})}=0$.

By a rational integral of bidegree $(r,s)$ we will understand an expression of the form 
 \begin{equation}\label{ratint}
F(x,p)=\frac{P(x,p)}{Q(x,p)}=\frac{P^{i_1\dots i_r}(x)p_{i_1}\cdots p_{i_r}}{Q^{i_1\dots i_s}(x)p_{i_1}\cdots p_{i_s}}, 
 \end{equation}
where $P$ and $Q$ are relatively prime homogeneous polynomials in $p$ of degrees $r$ and $s$, respectively.
Denote the (nonlinear) space of such integrals by $\mathcal{R}_{m,r,s}(g)$. It will be demonstrated that this space is 
a (singular, reducible) algebraic variety, so its dimension is well defined (maximal dimension among the components at general points).
We suggest to investigate the following.

\medskip

{\bf Conjecture.} {\it The space $\mathcal{R}_{m,r,s}(g)$ is finite-dimensional with the dimension bounded by 
that of $\mathcal{R}_{m,r,s}(g_0)$.}

\medskip

\noindent
Here and in what follows $g_0$ is a flat metric, which (by Remark \ref{Rk3}) can be changed to a metric of constant curvature.
Below we will motivate this claim by demonstrating finite-dimensionality of $\mathcal{R}_{m,r,s}(g)$ and showing that
$\dim\mathcal{R}_{m,r,s}(g_0)$ is actually $\Lambda_{m,r}+\Lambda_{m,s}-1$,
based on $\mathcal{R}_{m,r,s}(g_0)=\{P/Q:P\in\mathcal{K}_{m,r}(g_0),Q\in\mathcal{K}_{m,s}(g_0)\}$.

Note that if a polynomial in momenta function is an integral, so are all its homogeneous components.
For more general rational inhomogeneous integrals, where either numerator or denominator of such expression does not vanish at $p=0$,
one of the functions $F$ or $F^{-1}$ is analytic near the zero section in $TM$ and hence gives rise to Killing tensors by 
Whitteker's theorem \cite{Wh}, which was perhaps already known to Darboux \cite{D}. 
The following motivates consideration of rational homogeneous functions as in \eqref{ratint}.

 \begin{lemma}
Let $F=\frac{P}{Q}$ be an integral of $H$, where $P=\sum_{i=r_0}^{r_1}P_i$, $Q=\sum_{i=s_0}^{s_1}Q_i$, with $P_i,Q_i$ being homogeneous
polynomials in $p$ of degree $i$. Then both $F_0=\frac{P_{r_0}}{Q_{s_0}}$ and $F_1=\frac{P_{r_1}}{Q_{s_1}}$ are rational integrals.
 \end{lemma}
 
Warning: the case $r_0=s_0=0$ leads to the trivial integral, that's why it was considered before the lemma.
We assume that the extreme terms $P_{r_0},P_{r_1},Q_{s_0},Q_{s_1}$ in the Lemma are nontrivial. 

 \begin{proof}
We have the polynomial expression 
 $$
Q^2\{H,F\}=\{H,P\}Q-\{H,Q\}P=0.
 $$ 
Taking its homogeneous components of degrees $r_0+s_0+1$ and $r_1+s_1+1$ yields the claim.
 \end{proof}

It was proven by Ten in \cite{Te} (see also \cite{KoZ}) that for any degree $d$ there exists a (local) metric $g$ on $M$ 
with an irreducible Killing $d$-tensor (irreducibility means impossibility to represent it through simpler integrals;
for Killing tensors: of smaller degree). 
Similarly, Kozlov demonstrated in \cite{Ko} that there exist (local) metrics $g$ on $M$ with irreducible rational integrals of 
any bidegree $(r,s)$.

\subsection{Relative Killing tensors}

In order to underdstand the space of rational integrals let us introduce the following concept.

 \begin{definition}
A function $K=K(x,p)$ that is homogeneous polynomial in momenta of $\deg K=d$ is called a relative Killing tensor of degree $d$ for metric $g$ 
on $M$ (or more precisely $L$-relative Killing $d$-tensor) if 
 \begin{equation}\label{RK}
\{H,K\}=L\cdot K
 \end{equation} 
for some function $L$ (called cofactor) necessarily linear in momenta: $\deg L=1$.
 \end{definition}

This condition means that the Hamiltonian vector field $X_H$ is tangent to the submanifold $K=0$ in 
$T^*M\stackrel{g}\simeq TM$, and this partially constrains the dynamics. 
 
 \begin{rk}
An equivalent concept appeared previously under the name generalized Killing tensor in \cite{AHT} and
under the name Darboux polynomial in \cite{MP}. We prefer the above nomenclature, as it makes it analogous
to relative invariants for the group actions. Here the group $\R$ acts through the Hamiltonian flow of $H$. 
 \end{rk}

Note that if $K$ is a relative Killing $d$-tensor, then so is $e^\varphi K$ for any $\varphi\in C^\infty(M)$.
Indeed, in this case $L$ is changed to $L+d\varphi$ (if we understand $L$ as a one-form). In particular, $K$ is (locally)
conformal to a Killing $d$-tensor iff $dL=0$ (this two-form $dL$ is an analog of the Chern curvature in the web theory).
We will not distinguish between conformal solutions, and in what follows we will identify the pairs 
 \begin{equation}\label{gauge}
(K,L)\sim(e^\varphi K,L+d\varphi).
 \end{equation}
This will be called gauge equivalence.

 \begin{rk}
It is instructive to compare the above definition with the definition of conformal Killing tensor $K$, which is $\{H,K\}=L\cdot H$.
In this case scaling of $H$ (or equivalently $g$) keeps this condition invariant (modulo modification of $L$, which now becomes
a symmetric $(d-1)$-form).
 \end{rk}

Let us denote the space of $L$-relative Killing $d$-tensors of $g$ by $\mathcal{K}^L_{m,d}(g)$. 
Note that this is a linear space due to linearity of the defining PDE system, and we have 
$\mathcal{K}^{L-d\varphi}_{m,d}(g)=e^{\varphi}\mathcal{K}^L_{m,d}(g)$. 

 \begin{theorem}\label{first}
For any $L$ we have $\dim\mathcal{K}^L_{m,d}(g)\leq\Lambda_{k,d}$.
 \end{theorem}
 
 \begin{proof}
The defining equation \eqref{RK} for $L$-relative Killing $d$-tensor $K$ has the same symbol as that for $d$-Killing tensors
(we denote $\nabla_iK_\sigma=K_{\sigma;i}$, $\nabla_jK_{\sigma;i}=K_{\sigma;ij}$, etc)
 \begin{equation}\label{Ksi}
K_{(\sigma;i)}=L_{(i}K_{\sigma)}
 \end{equation}
in which $\sigma=\langle i_1\dots i_d\rangle$ is 
a multiindex of length $|\sigma|=d$ and $(\dots)$ denotes the symmetrization. Using jet-language, it defines the equation-submanifold $\mathcal{E}_d\subset J^1(S^dTM)$ with
empty complex characteristic variety, whence system \eqref{Ksi} has finite type (we refer to \cite{KL1}
for basics of the geometry of differential equations).
 
Prolongations $\mathcal{E}_d^{(k)}$ of this equation (obtained by left composition with total derivatives) is
 \begin{equation}\label{eqE}
K_{(\sigma;i)\varkappa}=\sum_{\nu\subset\varkappa}\binom{\varkappa}{\nu}L_{(i;\nu}K_{\sigma);\varkappa-\nu}\qquad
(|\varkappa|=k),
 \end{equation}
where $\varkappa$ is a multi-index and we assume symmetrization of the rhs applies only to $(i,\sigma)$.
In particular, for $k=d$ this system of equations is determined wrt $K_{\sigma;i\varkappa}$ and can be rewritten 
(see Lemma \ref{L3} below)
 $$
K_{\sigma;\tau}=F_{\sigma\tau}(\{L_{i;\nu},K_{\mu;\nu}:1\leq i\leq m,|\mu|=d,|\nu|\leq d\}) \qquad 
(|\sigma|=d,\ |\tau|=d+1).
 $$
This is a complete (Frobenius, finite type) system. Its space of solutions is parametrized by free jets
subject to the compatibility conditions. The latter may only decrease the fiber of this equation over $M$,
which is thus maximal for the Cauchy data read off the symbol. This corresponds to the flat space
$g=g_0$ (where $F_{\sigma\tau}=0$), in which case the dimension is $\dim\mathcal{K}^L_{m,d}(g_0)=\Lambda_{m,d}$.
 \end{proof}

 \begin{rk}\label{Rk3}
The last part also applies to any metric $g_0$ of constant sectional curvature (all such metrics are projectively equivalent). 
This is based on projective invariance of equation \eqref{Ksi}.

In the case $L=0$, this is known as the projective invariance of the Killing equation, see e.g.\ \cite{E,EM}.
Let us recall the argument. Consider the line bundle $K_M^w=(\Lambda^mT^*M)^{\frac{-w}{m+1}}$ of weight $w\in\R$
associated to the canonical bundle, and for any vector bundle $V$ denote $V[w]=V\otimes K_M^w$
the corresponding weighted bundle and the same for its sections $\Gamma(V)[w]$. 
For $\sigma\in\Gamma(K_M^w)$, under projective change of connection
$\hat\nabla_XY=\nabla_XY+\Upsilon(X)Y+\Upsilon(Y)X$ for some 1-form $\Upsilon$, we have 
$\hat\nabla\sigma=\nabla\sigma+w\Upsilon\otimes\sigma$, whence 
 $$
\hat\nabla_{i_0}\omega_{i_1\dots i_k}=\nabla_{i_0}\omega_{i_1\dots i_k}+(w-k)\Upsilon_{i_0}\omega_{i_1\dots i_k}
-\sum_{t=1}^k\Upsilon_{i_t}\omega_{i_1\dots i_{t-1}i_0i_{t+1}\dots i_k},\quad \omega\in\Gamma(\otimes^kT^*M)[w].
 $$
For $w=2k$ this yields $\hat\nabla_{(i_0}\omega_{i_1\dots i_k)}=\nabla_{(i_0}\omega_{i_1\dots i_k)}$, i.e.\
the (weighted) Killing operator $\nabla:\Gamma(S^kT^*M)[2k]\to\Gamma(S^{k+1}T^*M)[2k]$,
$K_\sigma\mapsto\nabla_{(i}K_{\sigma)}$, is projectively invariant.

The argument straightforwardly generalizes to the relative case. Indeed, if we treat $L$ in  \eqref{Ksi} un-weighted and let $\hat{L}=L-\Upsilon$,
then we have $\hat\nabla_{(i}K_{\sigma)}-\hat{L}_{(i}K_{\sigma)}=\nabla_{(i}K_{\sigma)}-L_{(i}K_{\sigma)}$.
 \end{rk}

The projective invariance implies that in the flat case ($g=g_0$, $L=0$) the space of Cauchy data
(equivalently: solution space $\mathcal{K}_{m,d}(g_0)$ is the representation $\Gamma_{d\cdot\pi_2}$ of the group $SL(m+1)$,
and exploiting the Weyl dimension formula we get $\dim\mathcal{K}_{m,d}(g_0)=\dim\Gamma_{d\cdot\pi_2}=\Lambda_{m,d}$.
(Here and below $\pi_k$ denotes the $k$-th fundamental weight of the Lie group $SL(m+1)$ and $\Gamma_{\nu}$ the representation of weight $\nu$.)
 
  
 \begin{rk}
It is known that the space $\mathcal{K}_{m,d}(g)$ is isomorphic to the space of tensors parallel wrt prolongation connection $D_g$ 
on the bundle $\mathcal{E}_d^{(d)}\subset J^{d+1}(S^dTM)$ over $M$. The rank of this bundle is precisely $\Lambda_{m,d}$.
Formulae \eqref{eqE} for $k\leq d$ determine a deformed connection $D_g^L$ such that its space of parallel sections is isomorphic
to $\mathcal{K}^L_{m,d}(g)$. Clearly the solution space is maximal iff this connection is flat.
We conjecture that this flatness implies the vanishing 
$dL=0$; this was verified for $d=2$.
 \end{rk}
 
The ring of relative Killing tensors is doubly graded by $d$ and $L$:
$\mathcal{K}^{L_1}_{m,d_1}(g)\cdot\mathcal{K}^{L_2}_{m,d_2}(g)\subset\mathcal{K}^{L_1+L_2}_{m,d_1+d_2}(g)$.
A similar relation holds for the corresponding field of rational fractions.

\subsection{Rational integrals with analytic coefficients}

Given $P\in\mathcal{K}^L_{m,r}(g)$ and $Q\in\mathcal{K}^L_{m,s}(g)$ with the same $L$, their ratio 
$F=\frac{P}{Q}$ belongs to $\mathcal{R}_{m,r,s}(g)$. 
We assume $P$ and $Q$ have no common (polynomial in $p$) factors. 

To discuss factors we assume till the end of this section that the metric $g$ and the coefficients 
of the rational functions $K$ are analytic in $x$ variables (in the rest of the paper those are smooth in $x$). 
We begin with the complex case ($M,g,K$ holomorphic) and comment on the real case at the end.

 \begin{lemma}
If $K\in\mathcal{K}^L_{m,d}(g)$, then $K|_\Sigma\not\equiv0$ (evaluation along $\Sigma$)
for any analytic variety $\Sigma\subset M$
of codimenson 1. Similarly, if $F\in\mathcal{R}_{m,r,s}(g)$, then $R|_\Sigma\not\equiv0,\infty$ for any analytic variety 
$\Sigma\subset M$ of codimenson~1. 
 \end{lemma}
 
 \begin{proof}
Consider the opposite case $K|_\Sigma\equiv0$. We may assume without loss of generality that $\Sigma$ is irreducible.
Morover the statement is localizable, so we may shrink $M$ if necessary and assume $\Sigma$ to be 
defined by one function. Thus let $\Sigma=\{a=0\}$ for $a\in\mathcal{O}(M)$, such that $da|_\Sigma\not\equiv0$. 

We can write $K=a^kK'$ for some $k\in\mathbb{N}$ and polynomial $K'$ with analytic coefficients such that 
$K'|_\Sigma\not\equiv0$. Then 
 $$
a^kLK'=LK=\{H,K\}=\{H,a^k\}K'+a^k\{H,K'\}=a^{k-1}\bigl(k\{H,a\}K'+\{H,K'\}a\bigr),
 $$
whence $\{H,a\}|_\Sigma=0$. Since $\{H,a\}$ is $da$ with $g$-raised indices, this is a contradiction.

The proof for a rational integral $F$ is similar.
 \end{proof}

To proceed let us recall that the sheaf of analytic functions $\mathcal{O}_M$ is Noetherian and UFD \cite{GR}, 
hence by the Hilbert theorem the same holds for the ring of polynomials with local coefficients $\mathcal{O}_M[p_1,\dots,p_m]$. 

With global coefficients the ring $A=\mathcal{O}(M)$ may fail the UFD property, yet it is factorial 
(i) in the setup of the global Weierstrass lemma \cite{GR}, (ii) for Cousin-II sets \cite{Coe}.
In any of the above cases a meromorphic function $f=\frac{a}{b}$, $a\in A$, $b\in A_\times$ can be represented by
a ratio $\frac{a}{b}$ with no common components of codimension one for $Z(a)$ and $Z(b)$. 
This also holds (iii) if $M$ is reduced to any smaller $M'\Subset M$
(then the number of components of the algebraic set $\{a=0\}$ for $a\in A$ is finite in $M'$).

Thus even though the ring $A[p_1,\dots,p_m]$ may fail to be Noetherian in general, it has no zero divisors and 
is factorial in any of the above cases, which we assume without further saying.
Note also that units in this ring are nowhere vanishing analytic functions, 
so $\op{gcd}(F,G)=1$ for $F,G\in A[p_1,\dots,p_m]$ means the polynomials $F,G$ have no common 
factors of positive degree in $p$ and no common factors from $A$ except for units.
 
 \begin{lemma}
Let $F,G,H\in\mathcal{O}(M)[p_1,\dots,p_m]$ and $\op{gcd}(F,G)=1$. 
Assume $F$ and $G$ do not vanish on a common analytic set of codimension one in $M$. If $F|(G\cdot H)$ then $F|H$.
 \end{lemma}

 \begin{proof}
Since the ring $A=\mathcal{O}(M)$ has no zero divisors, we may pass to the field of fractions $R=A/A_\times$, consisting of meromorphic functions on $M$ with global presentation as ratio $f=\frac{a}{b}$.
The corresponding ring $R[p_1,\dots,p_m]$ is Noetherian and UFD, and $\op{gcd}(F,G)=1$ still holds in it.
Therefore $H=FK$ for some $K\in R[p_1,\dots,p_m]$. Decompose $K=a^{-1}L$ for $L\in A[p_1,\dots,p_n]$, $a\in A_\times$.

The analytic set $\{a=0\}$ consists of at most countable number of components in $M$ (possibly with finite multiplicity). 
Each is contained in the common set of zeros for all coefficients of either $F$ or $L$. However if $F$ vanishes on a component, then
$H$ must do the same, so no cancellation comes from this part. Thus $L$ must vanish on all components, so it must be divisible by $a$
and we get the required factorization.
 \end{proof}

 \begin{lemma}
Let $K\in\mathcal{K}^L_{m,d}(g)$. Then every polynomial factor of $K$ belongs to some space $\mathcal{K}^{L'}_{m,d'}(g)$.
 \end{lemma}
 
 \begin{proof}
If $K=PQ$ with $\op{gcd}(P,Q)=1$ over $\mathcal{O}(M)[p_1,\dots,p_n]$ then
$\{H,K\}=\{H,P\}Q+\{H,Q\}P=LPQ$, so that $P$ divides $\{H,P\}$ and $Q$ divides $\{H,Q\}$.
Similarly, if $K=P^n$, then $P$ divides $\{H,P\}$.
 \end{proof}

 \begin{lemma}
Let $F=\frac{P}{Q}\in\mathcal{R}_{m,r,s}(g)$ with $\op{gcd}(P,Q)=1$. 
Then we get $P\in\mathcal{K}^L_{m,r}(g)$, $Q\in\mathcal{K}^L_{m,s}(g)$ for some $L\in\Gamma(TM)$.
 \end{lemma}
 
 \begin{proof}
Indeed, $\{H,F\}Q^2=\{H,P\}Q-\{H,Q\}P=0$, so $P$ divides $\{H,P\}$ and $Q$ divides $\{H,Q\}$.
The corresponding factors $L$ must be equal.
 \end{proof}
 
 For $L\in\Gamma(TM)$, let us denote by $\mathcal{S}^L_{r,s}$ the space of pairs $(P,Q)\in\mathcal{K}^L_{m,r}(g)\times\mathcal{K}^L_{m,s}(g)$
such that $P,Q\neq0$, $\op{gcd}(P,Q)=1$, and let $\mathcal{S}_{r,s}=\cup_L\mathcal{S}^L_{r,s}$.
(Note that we do not require $P$ and $Q$ to be irreducible, just relatively prime in $p$-variables.)
For $r=s$ this entails (but is not bounded to) $\dim\mathcal{K}^L_{m,r}(g)>1$.

Let us call $L$ admissible if the above set $\mathcal{S}^L_{r,s}$ is nontrivial.
It will be demonstrated below that the space of admissible cofactors $L$ modulo the gauge equivalence
is an algebraic variety (possibly reducible) so that the union above makes a perfect sense; 
in particular, it gives a stratification to the space of rational integrals of fixed bi-degree.

 \begin{cor}\label{cr1}
The following describes analytic rational integrals: $\mathcal{R}_{m,r,s}(g)=\bigl\{\tfrac{P}{Q}:(P,Q)\in\mathcal{S}_{r,s}\bigr\}$.
 \end{cor}
 
In particular, for fractional-linear integrals we have 
$\mathbb{P}\mathcal{R}_{m,1,1}(g)\simeq\cup_L
(\mathbb{P}\mathcal{K}^L_{m,1}\times\mathbb{P}\mathcal{K}^L_{m,1}\!\setminus\!\op{diag})$,
where the union is taken over such $L$ that $\mathcal{K}^L_{m,1}$ has $\dim>1$
(and the lhs means projectivization of the cone). 

We expect the above claim to hold also in the real analytic case, i.e.\ for relative Killing tensors  
$K\in C^\omega(M)[p_1,\dots,p_n]$. The subtelty is that the components of $\{a=0\}$ for $a\in C^\omega(M)[p_1,\dots,p_n]$ 
may be of higher codimension, so the arguments do not generalize straightforwardly. 
Yet by shrinking $M$ and using the complexification the results can be carried over (the simplest case
is when $M$ is a germ of a real analytic manifold).
The case $K\in C^\infty(M)[p_1,\dots,p_n]$ is even more complicated as this ring has zero divisors.

\section{Finite dimensionality of the space of rational integrals}\label{S2}

The following result is a direct extension of Theorem 1 of \cite{KM}; in fact it was briefly mentioned
in the concluding Section 4 there. 
We will give a short proof for completeness.

 \begin{theorem}\label{Tfri}
For any metric $g$ the space $\mathcal{R}_{m,r,s}(g)$ is finite-dimensional. 
 \end{theorem}

 \begin{proof}
Any integral is determined by its values in $\pi^{-1}(U)\subset TM$ for a small neighborhood $U\subset M$ via the geodesic flow. 
A rational integral $F\in\mathcal{R}_{m,r,s}(g)$ (with fixed $m,r,s$) at a point $a\in M$, i.e.\ on the tangent space $T_aM$,
is determined by a finite number $n$ of parameters, in fact 
 $$
n=\binom{r+m-1}r+\binom{s+m-1}s-1.
 $$
Since $F$ is homogeneous, it suffices to determine its values on the unit tangent bundle $T^1M$.

Let us choose $n$ points $a_1,\dots,a_n$ in general position in $U$ (we refer to \cite{KM} for details on this).
Choose a point $a\in M$. There are geodesics $\gamma_1,\dots,\gamma_n$ connecting $a$ to the selected points.
We assume unit length parametrization of all those geodesics (in the case of pseudo-Riemannian metric
this means the square length is $\pm1$ so we exclude null geodesics, and that holds for generic points $a$).

Since integrals are constant along geodesics, the values of $F$ on $\gamma_i$ at $a_i$ give the values of $F$ on $\gamma_i$ 
at $a$. This gives $n$ values of the rational integral $F$ on $T_aM$. For generic $a$ the position of vectors 
$\gamma_i(a)\in T^1_aM$ are generic, hence the integral $F$ is uniquely restored. Thus we obtain $F$ on a dense set 
in $T^1M$, whence by continuity $F$ is determined everywhere.
 \end{proof}
 
 \begin{rk}
The proof tells that $\mathcal{R}_{m,r,s}(g)$ is parametrized by a finite number of constants at different points,
thus it embeds this space of rational intergrals into a finite-dimensional Euclidean space as a smooth submanifold.
Collapsing those points to one enforces using more parameters at that point, namely higher jets of the coefficients of the integral
$F=P/Q$, as is discussed in the next section. This shows that $\mathcal{R}_{m,r,s}(g)$ is, actually, an algebraic affine variety.
 \end{rk}
 
This theorem gives the upper bound $n^2$ on the dimension of the space $\mathcal{R}_{m,r,s}(g)$ of rational integrals
of bi-degrees ($r,s)$, even those of low regularity. 
However this upper bound exceeds the one from the conjecture in the introduction. 
Even in the simplest case of fractional-linear integrals on surfaces $m=2$, $r=s=1$, we get $n=3$
and the bound $n^2=9$ exceeds the sharp bound 5 of \cite{AA}, see also Theorem \ref{th2D} below.

\section{Finite dimensionality of the space of relative Killing tensors}\label{S2+}

Equation $\mathcal{E}_d$ given by \eqref{Ksi} is a linear PDE for $K$, which is algebraic in $L$,
however its prolongations contains jets of $L$, so we may consider it as a system on $K,L$.
As such it contains $\binom{m+d}{d+1}$ equations on $\binom{m+d-1}d+m$ unknowns,
minus the gauge freedom that is 1 function, so roughly $\binom{m+d-1}d+m-1$ unknowns. 
Since $\binom{m+d-1}{d+1}>m-1$ for $m>2$ the system $\mathcal{E}_d$ appears overdetermined except for $m=2$.
As shown in \cite{Kr2} for $m=2$, $d=1$ the system is determined, and this generalizes to
the case $d>1$.

Our main interest is when this system has more than one solution $K$ for the same $L$.
Let us enhance \eqref{Ksi} with the constraint $\dim\mathcal{K}^L_{m,d}>1$, and denote the obtained system of equations
by $\hat {\mathcal{E}}_d$ for a fixed $d$; now this is a system on both $K$ and $L$ and it is algebraic nonlinear. 

This system $\hat {\mathcal{E}}_d$ is not compatible, and adding the compatibility conditions (that is, 
bringing the system to involution) enhances $\hat{\mathcal{E}}_d$ with new equations that are already differential in $L$. 
We claim that this system is over-determined and moreover, modulo gauge, is of finite type.
(This means we get a finite type system by ``fixing gauge'' 
even though this may be non-invariant, for instance coordinate-wise.)

 \begin{theorem}\label{thm2}
The space of solutions $(K,L)$ of system $\hat {\mathcal{E}}_d$, 
considered modulo conformal rescalings \eqref{gauge}, is finite-dimensional.
 \end{theorem}
 
Let us note that this statement is actually equivalent to Theorem \ref{Tfri} due to the correspondence between
rational integrals and relative Killing tensors, given by Corollary \ref{cr1}.
 
We would like to give another proof using relative Killing tensors, which reveals algebraic structure of the space of rational integrals. 
For this let us give more details on equations \eqref{eqE} having the form
 \begin{equation}\label{prEq}
K_{(\sigma;i)\varkappa}=L_{(i;\varkappa}K_{\sigma)}+(lots) 
 \end{equation}
where symmetrization in the rhs only concerns $(i,\sigma)$ and $(lots)$ stands for ``lower order terms'' wrt $L$
(and also $K$ from the lhs). Resolution of those by the top-jets of $K$ has the following form
(below $[...]$ stands for skew-symmetrization by the indicated indices):

Denote by $\sigma=\sigma'\sqcup\sigma''$ an ordered splitting, i.e.\ a disjoint union such that $i'<i''$ $\forall$ $i'\in\sigma'$,
$i''\in\sigma''$. 

 \begin{lemma}\label{L3}
For multi-indices $\sigma=\sigma_1\dots\sigma_d$, $\tau=\tau_0\tau_1\dots\tau_d$ of lengths $|\sigma|=d$, $|\tau|=d+1$,
denote $\op{Sym}_\sigma$ and $\op{Sym}_\tau$ the operators of symmetrization by those indices.
Then we have
 \begin{equation}\label{Ktopjet}
K_{\sigma;\tau}=L_{(\tau_0;\tau_1\dots\tau_d)}K_\sigma-2\op{Sym}_\sigma\op{Sym}_\tau\Bigl\{\sum_{k=1}^d(-1)^k
\binom{d}{k}\sum L_{[\sigma_i;\tau_j]\sigma'\tau'}K_{\sigma''\tau''}\Bigr\}+(lots),
 \end{equation}
where the last sum is over $\sigma_i\in\sigma$, $\tau_j\in\tau$ with 
$\sigma\setminus\sigma_i=\sigma'\sqcup\sigma''$, $\tau\setminus\tau_j=\tau'\sqcup\tau''$
and $k=|\tau''|=d-|\sigma''|$.
 \end{lemma}

We can make $(lots)$ explicit, but will not need it.
For instance, in the case $d=1$ the $\langle ijl\rangle$ equation is
 $$
K_{i;jl}+K_{j;il}=L_{i;l}K_j+L_{j;l}K_i+L_iK_{j;l}+L_jK_{i;l}
 $$
and adding to it the equation for $\langle ilj\rangle$ and subtracting the equation for $\langle jli\rangle$ we get
 $$
K_{i;(jl)}=L_{(j;l)}K_i+L_{[i;j]}K_l+L_{[i;l]}K_j+L_iL_{(j}K_{l)}-\tfrac12(R_{ij}{}^k{}_l+R_{il}{}^k{}_j)K_k.
 $$
 
 \begin{proof}
It is well-known \cite{Th,Wo} that the $d$-th prolongation of any system 
$K_{(\sigma;i)}=\Psi_{\sigma,i}(K_{\tau}:|\tau|=d)$ over all $\sigma,i$ with $|\sigma|=d$ can be resolved wrt
highest derivatives $K_{\sigma;\tau}$, where $|\sigma|=d$, $|\tau|=d+1$. System \eqref{Ksi} is of this kind.
To verify the formula, one just needs to recheck that the partial symmetrization 
$K_{(\sigma_1\dots\sigma_d;\tau_0)\tau_1\dots\tau_d}$ in the left hand size of this formula gives 
what is stated in \eqref{prEq} for $i=\tau_0$.

To check this one has to compare the coefficients of $L_{i;\alpha}K_\beta$ for various $i,\alpha,\beta$,
$|\alpha|=|\beta|=d$, in the lhs and rhs. For the term
$L_{\tau_0;\tau_1\dots\tau_d}K_{\sigma_1\dots\sigma_d}$ or the terms obtained by a permutation
of $\sigma_1\dots\sigma_d\tau_0$ this coefficient is 
 $$\frac1{d+1}\bigl(\frac1{d+1}+2d\frac12d\frac1d\frac1{d+1}\bigr)=\frac1{d+1}$$
as expected (the factors are due to symmetrization and skew-symmetrization;
note that in the rhs only the case $k=0$ of the first term and $k=1$ of the following terms occur).

For the term $L_{\tau_0;\sigma'\tau'}K_{\sigma''\tau''}$ with $|\sigma'|=|\tau''|=p>0$, $|\sigma''|=|\tau'|=d-p$,
and for other terms obtained by the same permutation, the coefficient is
 \begin{multline*}
\hskip24pt
\frac2{d+1}\left((-1)^{d-p+1}\binom{d}{p}\cdot\tfrac{-1}2\cdot\frac1{\binom{d}{p}}\cdot\frac1{\frac{(d+1)!}{(d-p)!p!}}
+ (-1)^{d-p+1}\binom{d}{p}\cdot(d-p)\cdot\tfrac12\cdot\frac1{\frac{d!}{(d-p-1)!p!}}\cdot\frac1{\binom{d+1}{p+1}}
\right.\\
\left.+ (-1)^{d-p}\binom{d}{p-1}\cdot p\cdot\tfrac12\cdot\frac1{\frac{d!}{(d-p)!(p-1)!}}\cdot\frac1{\binom{d+1}{p}}
\right)=0.\hskip24pt
 \end{multline*}Similarly one checks the coefficients of $L_{\tau_1;\sigma'\tau'}K_{\sigma''\tau''}$ where $\tau_0\in\tau'$
and the terms obtained by a permutation of $\sigma_1\dots\sigma_d\tau_0$, etc, which again give 0, as expected.
This proved the claim.
 \end{proof}

Now the idea to prove the theorem is as follows. Lemma \ref{L3} expresses higher derivatives of $K$, but $L$ is still 
un-constrained. We explore the compatibility conditions to express higher derivatives of $L$ (modulo gauge).

 \begin{proof}[Proof of Theorem \ref{thm2}]
Let us prolong system \eqref{Ktopjet} one more time, i.e.\ apply covariant derivative $\nabla_s$ to it:
 \begin{equation}\label{prol+}
K_{\sigma;\tau s}=
L_{(\tau_0;\tau_1\dots\tau_d)s}K_\sigma-2\op{Sym}_\sigma\op{Sym}_\tau\Bigl\{\sum_{k=1}^d(-1)^k
\binom{d}{k}\sum L_{[\sigma_i;\tau_j]s\sigma'\tau'}K_{\sigma''\tau''}\Bigr\}
+(lots) 
 \end{equation}
This system is symmetric in $\tau s$ modulo $(lots)$, and this gives the following type relation on $L$:
 \begin{equation}\label{compL}
\sum r_{i,j,\sigma,\tau}L_{[i;j]\tau}K_\sigma=(lots)
 \end{equation}
with the sum over $i,j$, multi-indices $\tau$, $\sigma$, and certain combinatorial coefficients $r_{i,j,\sigma,\tau}$, 
which are explicit from \eqref{prol+}.
If the relative Killing tensor fields span the whole $S^dTM$ at almost any point, then one can see that this implies 
the following system for any indices $i,j$ and multiindex $\tau$, $|\tau|=d$:
 \begin{equation}\label{eqL}
L_{[i;j]\tau}= G_{ij\tau}(\{L_{k;\nu},K_{\mu;\nu}:1\leq k\leq m,|\mu|=d,|\nu|\leq d\}).
 \end{equation}
Thus we see that system (\ref{eqE})+(\ref{eqL}), which is the first step in completion to involution of \eqref{eqE}, is
overdetermined. We claim that modulo the conformal rescale freedom it is of finite type.

To demonstrate this let us compute the symbol of the system. Recall (cf.\ \cite{KL1})
that for $i$-th equation of the system this is obtained by 
(i) linearizing the system, (ii) dropping lower order terms, (iii) changing derivation $\p_k$ of the $j$-th unknown to the 
corresponding momenta coordinate $p_k$ -- this contributes to the entry $(i,j)$ of the matrix.
In our case the highest order terms of $K$ are contained in \eqref{eqE} while the highest order terms of $L$ are contained 
in \eqref{eqL}. Thus the symbol matrix has block form
 $$
\begin{pmatrix}A & 0\\ 0 & B\end{pmatrix}.
 $$ 
Here $A$ is a matrix of size $\binom{m+d}{d+1}\times\binom{m+d-1}{d}$ consiststing of polynomials in $p$
of degree $d+1$, it corresponds to the symbol of the Killing equation for $d$-tensors; 
whereas $B$ of size $\binom{m}2\cdot\binom{m+d-1}{d}\times m$
consists of rows of the form $v\cdot p_\mu$, where $|\mu|=d$ and $v$ is linear in $p$,
spanning the orthogonal complement to the vector $p=(p_1,\dots,p_m)$ wrt the dot-product.

Recall that the system is of finite type iff the complex projective characteristic variety is empty, 
i.e.\ the affine variety is 0, cf.\ \cite{KL1}.
The affine characteristic variety consists of covectors $p$, where the rank of the above symbol matrix drops. 
In our case this is equivalent to drop of the rank for either $A$ or $B$. But $A$ corresponds to the Killing equation, which is
of finite type, meaning that for (complex) $p\neq0$ the rank is maximal. On the contrary $B$ is always of rank $<m$ 
because the vector-column $p$ is in its kernel.

The latter means that every covector is characteristic, however this was expected, as the conformal change freedom
$K\mapsto e^\varphi K$ implies that the solution space depends on (at least) a function $\varphi$ of $m$ arguments. 
Yet we may eliminate this freedom.
Indeed, system \eqref{eqL} can be considered as a complete system  of order $d$ on the coefficients of $dL$, which 
completely resolves the gauge freedom $L\mapsto L+d\varphi$. With this, the characteristic variety becomes trivial, and consequently this system has finite type.

In the case we have non-maximal set of relative Killing vectors at generic point of $M$, the system
may become of infinite type. Indeed, this is precisely what happens for $\dim\mathcal{K}_{m,d}^L=1$.
However with the condition $\dim\mathcal{K}_{m,d}^L>1$ specifying equation $\hat {\mathcal{E}}_d$
the system is still of finite type. 

This is a lengthy argument, for simplicity, let us make it explicit in the case $d=1$, where equation \eqref{Ktopjet} becomes
 $$
K_{a;bc}=K_aL_{(b;c)}+2\op{Sym}_{bc}\bigl\{L_{[a;c]}K_b\bigr\}+(lots)
 $$
and the corresponding relations \eqref{compL} are
 $$
K_aL_{[c;d]b}-K_bL_{[c;d]a}+K_cL_{[a;b]d}-K_dL_{[a;b]c}\doteq0,
 $$
where by $\doteq$ we denote equality modulo $(lots)$. This implies the equations
 $$
K_aL_{[c;b]b}+K_cL_{[a;b]b}-K_b(L_{[a;b]c}+L_{[c;b]a})\doteq0 
 $$
when $d=b$, and furthermore when $c=a$ we get
 $$
K_aL_{[a;b]b}-K_bL_{[a;b]a}\doteq0
 $$
and we also may use the identity $L_{[a;b]c}+L_{[b;c]a}+L_{[c;a]b}\doteq0$.

Assume now that $K_1$ and $K_2$ are independent non-zero components at a point $x\in M$, while the other components
$K_i$ may vanish (this comes without loss of generality by applying a coordinate transformation). 
By choosing various indices in the above identities we obtain for all $i,j,k$
 $$
L_{[i;j]k}\doteq0,
 $$
which clearly yields finite type for $dL$ and hence finite type for $L$ modulo gauge.
 
In the case $d=2$ equation \eqref{Ktopjet} becomes
 $$
K_{ab;cde}=K_{ab}L_{(c;de)}+4\op{Sym}_{ab}\op{Sym}_{cde}\bigl\{L_{[a;c]d}K_{be}
-2\op{Sym}_{ab}\op{Sym}_{cde}\bigl\{L_{[a;c]b}K_{de}\bigr\}+(lots)
 $$
and the corresponding relations \eqref{compL} are
 \begin{gather*}
2K_{ab}L_{[e;f]cd}-K_{ac}L_{[e;f]bd}-K_{ad}L_{[e;f]bc}-K_{bc}L_{[e;f]ad}-K_{bd}L_{[e;f]ac}+2K_{cd}L_{[e;f]ab}\\
-K_{ae}(L_{[c;b]df}+L_{[d;b]cf})-K_{be}(L_{[c;a]df}+L_{[d;a]cf})
+K_{af}(L_{[c;b]de}+L_{[d;b]ce})+K_{bf}(L_{[c;a]de}+L_{[d;a]ce})\\
-K_{de}(L_{[a;c]bf}+L_{[b;c]af})-K_{ce}(L_{[a;d]bf}+L_{[b;d]af})
+K_{df}(L_{[a;c]be}+L_{[b;c]ae})+K_{cf}(L_{[a;d]be}+L_{[b;d]ae})\doteq0.
 \end{gather*}
From these equations and its degenerations we can again obtain a finite type system for $dL$.
In the case of general $d$, we can deduce the equation
 $$
L_{[i;j]k_1\dots k_{d-1}}\doteq0,
 $$
which yields finite type for $L$ modulo gauge.
 \end{proof}

 \begin{rk}
Using solutions $(K,L)$ of $\hat{\mathcal{E}}_d$ we can only get rational integrals in $\mathcal{R}_{m,r,s}(g)$ with $r=s=d$.
More generally, by the same technique one can prove that the equation describing admissible $L$, 
i.e.\ such that $\mathcal{S}^L_{r,s}$ is nontrivial,
together with the corresponding $(P,Q)\in\mathcal{S}^L_{r,s}$, is of finite type.
 \end{rk}
 
Let us mention an alternative approach to the proof.
In the part, where we derive compatibility conditions of \eqref{eqE}, note that
symbol of the equivalent system \eqref{Ksi} is the representation $\Gamma_{(d+1)\pi_1}$ of $SL(m)$. 
Its $k$-th prolongation (equiv: symbol of the prolongation)
is the representation $\Gamma_{(d+1)\pi_1}\otimes\Gamma_{k\pi_1}$. For $k=d$ this coincides with the space 
$\Gamma_{(d+1)\pi_1}\otimes\Gamma_{d\cdot\pi_1}\simeq\Gamma_{d\cdot\pi_1}\otimes\Gamma_{(d+1)\pi_1}$
of $(d+1)$-st derivatives of $\{K_\sigma:|\sigma|=d\}$. 

When we prolong one more time we get representation $\Gamma_{(d+1)\pi_1}\otimes\Gamma_{(d+1)\pi_1}$
that splits into $\Gamma_{(d+2)\pi_1}\otimes\Gamma_{d\cdot\pi_1}$, corresponding to 
$(d+2)$-nd derivatives of $\{K_\sigma:|\sigma|=d\}$, and the module $\Gamma_{(d+1)\pi_2}$.
This latter represents the compatibility conditions, annihilating the lhs of \eqref{Ksi} and giving equations \eqref{eqL}
when applied to the rhs. 

 \begin{cor}
The space of admissible $L$ (up to equivalence) is an algebraic set. 
 \end{cor}

 \begin{proof}
The equation $\hat{\mathcal{E}}_d$ on $(K,L)$ considered in Theorem \ref{thm2} is algebraic, and so is its union with 
algebraic constraint and completion to involution $\bar{\mathcal{E}}_d$. Since the solution space is finite-dimensional, 
every (local) solution is uniquely given by its $N$-jet at a fixed point $o$ (for some large $N$). 
Projecting admissible jets $(j^N_oK,j^N_oL)$
to the second component, we obtain an algebraic set of admissible jets of $L$. The equivalence relation
can be resolved in passing from $L$ to $dL$, which is also algebraic in parameters on the solution space.
 \end{proof}

\section{Computation for spaceforms}\label{S3}

Consider now a metric $g_0$ of constant sectional curvature. Since such metrics are projectively equivalent
the count of the dimension of integrals does not depend on the curvature sign.

 \begin{theorem}
$\mathcal{R}_{m,r,s}(g_0)=\{P/Q:P\in\mathcal{K}_{m,r}(g_0),Q\in\mathcal{K}_{m,s}(g_0)\}$. 
 \end{theorem}

 \begin{proof}
Let us start with the simplest case $m=2$ and the Euclidean metric $H=\tfrac12(p_1^2+p_2^2)$.
In this case we have three linear integrals (Killing vectors) $K_1=p_1$, $K_2=p_2$ and $K_3=x_1p_2-x_2p_1$,
which are functionally (and hence algebraic) independent. In fact they are independent in the complement to 
the null cone $2H=K_1^2+K_2^2=0$ (this is nontrivial in complexification but is the zero section in the real case).
If we introduce $J=x_1p_1+x_2p_2$ then transformation between coordinates $(p_1,p_2,x_1,x_2)$ and
$(K_1,K_2,K_3,J)$ is an algebraic diffeomorphism outside $H=0$, and we can pushforward vector fields. 

In particular, the vector field $\xi=p_1\p_{p_1}+p_2\p_{p_2}$ corresponds to $K_1\p_{K_1}+K_2\p_{K_2}+K_3\p_{K_3}+J\p_J$.
Let us show how this implies that all Killing tensors are combinations of Killing vectors. Any (smooth) first integral can be expressed 
as $F=F(K_1,K_2,K_3)$. If it is a Killing $d$-tensor, it also satisfies the condition $\xi(F)=d\cdot F$, which means
that $F$ is a homogeneous polynomial of degree $d$ in $K_1,K_2,K_3$.

If $F=P/Q$ is a rational function with $\deg P=r$, $\deg Q=s$, then expanding $P$ and $Q$ in variables $K_1,K_2,K_3,J$
we notice that $P=\tilde{P}(K_1,K_2,K_3)R(K_1,K_2,K_3,J)$ and $Q=\tilde{Q}(K_1,K_2,K_3)R(K_1,K_2,K_3,J)$, 
where $\tilde{P},\tilde{Q},R$ are such functions that 
$\xi(\tilde{P})=\tilde{r}\cdot\tilde{P}$, $\xi(R)=l\cdot R$, $\xi(\tilde{Q})=\tilde{s}\cdot\tilde{Q}$ and 
$\tilde{r}=r-l$, $\tilde{s}=s-l$ for some integer $l$.
Thus $\tilde{P}$ and $\tilde{Q}$ are polynomials in $K_1,K_2,K_3$ and $F=\tilde{P}/\tilde{Q}$ is a rational combination 
of Killing tensors. In fact, since $R$ is also a polynomial in $J$, it follows that for irreducible representation $F=P/Q$ we have 
no common factors and hence $\tilde{r}=r$, $\tilde{s}=s$.

Now consider the Euclidean metric $g_0$ in general dimension $m$: $H=\tfrac12\sum_1^mp_i^2$.
Basic Killing vectors are $K_i=p_i$ and $K_{ij}=x_ip_j-x_jp_i$ with basic syzygies $K_iK_{jl}+K_jK_{li}+K_lK_{ij}=0$.
Thus while there are $\binom{m+1}2$ linearly independent Killing vectors, only $2m-1$ of them are functionally and
hence algebraically independent. Near every point we can choose an index $k$ so that $K_i$ and $K_{kj}$ for $j\neq k$
are algebraically independent. Adding $J=\sum_1^mx_ip_i$ we get a coordinate system related to $p_i,x_i$ via an algebraic
diffeomorphism. 

Under this map, the vector field $\xi=\sum_1^mp_i\p_{p_i}$ corresponds to 
$\sum_{i=1}^nK_i\p_{K_i}+\sum_{j\neq k}K_{kj}\p_{K_{kj}}+J\p_J$. Now the rest of argumetns goes precisely as for $m=2$,
and we conclude here also that rational integrals $F=P/Q$ are generated by Killing vectors.

Finally, we note that the important feature of the flat case was the resonant behavior (superintegrability: $2m-1$ functionally 
independent integrals affine in momenta) that also holds for other spaces of constant curvature, whence the conclusion.
 \end{proof}
 
 \begin{cor}
$\dim\mathcal{R}_{m,r,s}(g_0)=\Lambda_{m,r}+\Lambda_{m,s}-1$.
 \end{cor}

\section{The case of a Riemann surface}\label{S4}

For dimension $m=2$ we can choose local isothermal coordinates, in which $g=e^{2\lambda(x,y)}(dx^2+dy^2)$
or equivalently $H=\tfrac12e^{-2\lambda(x,y)}(p^2+q^2)$, where we denote $p=p_1$ and $q=p_2$.

It was proved by Agafonov and Alves \cite{AA} that for Lorentzian metrics $g$ on a surface the dimension of 
$\mathcal{R}_{2,1,1}(g)$ can be either 5 or 3 (or 0 if this space is empty). 
Below we will prove this in Riemannian signature by a different method.
Note that even though the Riemannian and Lorentzian cases are related by a Wick rotaiton, the problems
of rational integrals (complexification of a nonlinear functions) are not readily equivalent,
and so our result does not directly follow from that of \cite{AA}.

 \begin{theorem}\label{th2D}
The space $\mathcal{R}_{2,1,1}(g)$ of (local) fractional-linear integrals of a metric $g$, if non-empty, has dimensions 5 or 3.
In the former case $g$ has constant curvature, the cofactor is trivial $L\sim0$ and the space of solutions is connected. 
In the latter case the space of solutions has finitely many components.
 \end{theorem}

 \begin{proof}
A relative Killing vector $u(x,y)\p_x+v(x,y)\p_y$ corresponds to the function $K=u(x, y)p + v(x, y)q$ on $T^*M$.
The cofactor via gauge $L\sim L+d\varphi$ can be chosen in the form $L=e^{-2\lambda(x,y)}a(x,y)p$.
Then the defining relation $\{H,K\}=LK$, split by variables $p,q$ is an overdetermined system
 \begin{equation}\label{qqq}
u_x= -(a+\lambda_x)u - \lambda_yv,\ v_x+u_y = -av,\ v_y = - \lambda_xu - \lambda_yv.
 \end{equation}
Its first compatibility condition may be calculated as the multi-bracket of \cite{KL2}, see also \cite{Kr1} for 
an example of such computation (equivalently: prolong to 3rd jets and find a syzygy). The result is 
 \begin{gather}\label{gap}
3(u_y-v_x)a_y+
\bigl(2\lambda_ya_y - 4(\lambda_{xx}+\lambda_{yy})\lambda_x + 2(\lambda_{xx}+\lambda_{yy})_x + 2a_{yy}\bigr)u \\
- \bigl((3a+2\lambda_x)a_y + 4(\lambda_{xx}+\lambda_{yy})\lambda_y  - 2(\lambda_{xx}+\lambda_{yy})_y +2a_{xy}\bigr)v=0.\notag
 \end{gather}
  
Thus if $a_y\neq0$ (equiv: $dL\neq0$) the above equation is nontrivial, and we get a complete system of the first order on $u,v$,
namely in addition to three equations of \eqref{qqq} we supply the following one
 \begin{equation}\label{ppp}
u_y = \frac1{3a_y}\Bigl( (e^{2\lambda}K_x-\lambda_ya_y-a_{yy})u + (e^{2\lambda}K_y+\lambda_xa_y+a_{xy})v) \Bigr),
 \end{equation}
where $K=-e^{-2\lambda}\Delta\lambda$ is the Gaussian curvature (half of the scalar curvature $R$).
 
In order for the fractional-linear integral to exist, system \eqref{qqq} should have at least two linearly independent solutions,
so system \eqref{qqq}+\eqref{ppp} must be compatible. 
The compatibility is the equality of mixed derivatives $u_{xy}=u_{yx},v_{xy}=v_{yx}$ modulo \eqref{qqq}+\eqref{ppp},
split by the coefficients of $u,v$, altogether four conditions. These have the form of the following system on $w=a_y\neq0$ 
(with this we completely remove the gauge freedom):
 \begin{gather*}
w_{xx}=U_{11}(x,y,w,w_x,w_y),\ \ w_{xy}=U_{12}(x,y,w,w_x,w_y),\ \ w_{yy}=U_{22}(x,y,w,w_x,w_y),\\
V_1(x,y,w)w_x+V_2(x,y,w)w_y+V_0(x,y,w)=0,
 \end{gather*}
where $U_i,V_j$ depend also on $\lambda$ and its derivatives in algebraic way
(the precise form of those functions is not important and hence not indicated). 
Thus the system is of finite type and completing it to involution we conclude that its solution space is algebraic.

The ratio of the obtained two relative Killing vectors is a fractional-linear integral $F=K_1/K_2$.
The group $PGL_2$ acts by reparametrization on the solution space 
 $$
\frac{K_1}{K_2}\mapsto\frac{a_{11}K_1+a_{12}K_2}{a_{21}K_1+a_{22}K_2}
 $$
and hence each component of $\mathcal{R}_{2,1,1}(g)$ is its orbit.

Next, if $a_y=0$ then the compatibility condition \eqref{gap} (vanishing of the coefficients of $u,v$) 
is equivalent to $R_x=0,R_y=0$, i.e.\ constancy of the scalar curvature $R$ of $g$. In this case system \eqref{qqq} is compatible, 
$L\sim0$ and we get 5-parametric family of solutions $u,v$. Namely each integral is the ratio of two Killing vectors:
$F=K_1/K_2$ (by simulteneous scaling of $K_1,K_2$ we can achieve $L=0$).

Finally, if $a_y=0$ but the scalar curvature $R$ is nonconstant, then we get an equation of order 0 in the system, its
prolongation allows to express all first jets of $u,v$, and hence the solution space $K_{2,1}^L(g)$ is at most 1-dimensional,
which does not allow to produce a fractional-linear integral. Hence such $L$ does not contribute to the space 
$\mathcal{R}_{2,1,1}(g)$ of rational integrals.
 \end{proof}

This confirms the conjecture, stated in the introduction, for the simplest case $m=2$, $r=s=1$.

%

\section{Examples}\label{S5}

{\bf 1.} Consider the metric $g=(x^2+4y^2)(dx^2+dy^2)$ on $\R^2(x,y)$. The corresponding Hamiltonian is 
 $$
H=\frac{p^2+q^2}{2(x^2+4y^2)}. 
 $$
There are no Killing vectors, but this metric is Liouville \cite{Kr1} and hence quadratic integrable.
In fact, it has 3 quadratic integrals, namely the energy and the following two:
 $$
F_1=\frac{(2yp-xq)(2yp+xq)}{x^2+4y^2},\quad
F_2=\frac{(2yp-xq)(x^2p+2y^2p-xyq)}{x^2+4y^2}.
 $$
In addition it has one cubic integral, which is automatically irreducible:
 $$
F_3=\frac{(2yp-xq)(2xp^2-2ypq+xq^2)}{x^2+4y^2}.
 $$
All quartic integrals are reducible, and apparently the same concerns the fifth degree integrals (verified only
for those divisible by $2yp-xq$). Of rational integrals we mention the following:
 $$
G_1=\frac{x^2p+2y^2p-xyq}{2yp+xq},\quad 
G_2=\frac{2xp^2-2ypq+xq^2}{x^2p+2y^2p-xyq}. 
 $$
 
 \begin{rk}
The above $H$ is equivalent to the Hamiltonian from Example 2 of \cite{AS} via the transformation
$(x,y)\mapsto(\arctan(y/x),\tfrac12\ln(x^2+y^2))$. The fractional-linear integral $F$ found there
corresponds to our $G_1$.
 \end{rk}

Let us now explain a relation of those integrals to relative Killing tensors. 
Consider the following functions
 $$
R_0=x^2p+2y^2p-xyq,\
R_1=2yp-xq,\
R_2=2yp+xq,\
R_3=2xp^2-2ypq+xq^2.
 $$
A direct verification shows that 
 $$
\{H,R_0\}=L_-R_0,\ \{H,R_1\}=L_+R_1,\ \{H,R_2\}=L_-R_2,\ \{H,R_3\}=L_-R_3,
 $$
where 
 $$
L_-=-\frac{6(xp+2yq)}{(x^2+4y^2)^2},\ L_+=\frac{2(xp-2yq)}{(x^2+4y^2)^2}. 
 $$
However since $\{H,-\tfrac14\ln(x^2+4y^2)\}=\tfrac{xp+4yq}{(x^2+4y^2)^2}$ these functions are cohomologous to 
 $$
L_-\sim -L,\ L_+\sim +L\ \text{ for }\ L=\frac{3xp}{(x^2+4y^2)^2}.
 $$
Now we can see that $R_0R_1$, $R_1R_2$, $R_1R_3$ with proper factors $f(x,y)$ give polynomial integrals, 
while $R_0/R_2$ and $R_3/R_2$ in a similar fashion give rise to rational integrals.
That is precisely how integrals $F_1,F_2,F_3$ and $G_1,G_2$ arise from those relative Killing vectors and tensors.

In particular, we notice that $G_1=F_2/F_1$, so the rational integral of Agapov and Shubin \cite{AS} is not a genuine
irreducible integral, but a ratio of Killing tensors.

\smallskip

{\bf 2.} Next consider the metric $g=(x^4+4y^4)(dx^2+dy^2)$, which corresponds to Example 3 from \cite{AS}.
Again, there are no Killing vectors, but the metric is Liouville. However now there are only two quadratic integrals

 $$
2H=\frac{p^2+q^2}{x^4+4y^4},\quad
F_2=\frac{(2y^2p-x^2q)(2y^2p+x^2q)}{x^4+4y^4}
 $$
and no cubic integrals in this case. 
A direct verification shows that the following is a basis of relative Killing vectors 
of the type $P(x,y)p+Q(x,y)q$ with $\op{deg}P,\op{deg}Q\leq3$:
 $$
R_1=2y^2p-x^2q,\ R_2=2y^2p+x^2q,\ R_3=(x^2+2y^2)yp+x^3q,\ R_4=(x^2-2y^2)p-2xyq;
 $$
The corresponding cofactors $L=\{H,\log R\}$ are:
 $$
L_1= \frac{-2(x-2y)(x^2p-2y^2q)}{(x^4+4y^4)^2},\ L_2= \frac{-2(x+2y)(x^2p+2y^2q)}{(x^4+4y^4)^2},\ 
L_3=L_4= -\frac{4x(x^2+y^2)p+2y(x^2+6y^2)q}{(x^4+4y^4)^2}.
 $$
One may note that $-L_1\sim L_2\not\sim L_3$.
We conclude that $F_2\sim R_1R_2$ and 
 $$
F_3=\frac{R_3}{R_4}= \frac{(x^2+2y^2)yp+x^3q}{(x^2-2y^2)p-2xyq}
 $$
is a fractional-linear integral (established in different coordinates in \cite{AS}). 
It is algebraically independent of the energy and $F_2$, however the Hamiltonian system
is super-integrable and there is still a possibility that $F_3$ may be a ratio of higher degree Killing tensors.
Thus this integral is irreducible in the sense of Kozlov but may still be reducible, like the integral from the previous example.

\smallskip

{\bf 3.} Let $J_0,J_1$ be the Bessel functions of the first kind. The following Hamiltonian 
 $$
H=e^{-2x}\frac{p^2+q^2}{J_0(y)^2+J_1(y)^2}. 
 $$
considered in \cite{AS} was found to possess a fractional-linear integral
 $$
F=\frac{(xp+yq)J_1(y) - (yp-xq)J_0(y)}{J_1(y)p + J_0(y)q},
 $$
and in \cite{Kr2} it was shown to be unique such modulo M\"obius trasformations. 

The system possesses neither linear nor quadratic integrals.
We did not exlore higher degree integrals, but if they are absent then $F$ will be an irreducible rational integral.
For this it would be sufficient to check whether the Hamiltonian system $H$ is non-resonant (that is,
almost every Liouville torus is the closure of every trajectory on it) which can be verified by analytic arguments. 

Note that the metrics $g$ from this and two previous examples 
give an instance of Theorem \ref{th2D}, for which the space of fractional-linear integrals
is 3-dimensional: the orbit of $PGL_2$-action on $G_1$ is $\frac{a_{11}F_1+a_{12}F_2}{a_{21}F_1+a_{22}F_2}$ for 
$a_{11}a_{22}-a_{12}a_{21}\neq0$. 

\smallskip

{\bf 4.} Finally let us mention the Hamiltonian from \cite{MP} and its higher-dimensional version \cite{AHT}
that we may further generalize to a system on $T^*\R^n$ as follows:
 $$
H=\tfrac12 g^{ij}p_ip_j+(b^ip_i)\cdot(m_1x^1p_1+\dots+m_nx^np_n),
 $$
where $(g^{ij})$ and $(b^i)$ are constant matrix and vector and $m_i$ are integers.

This system possesses rational integrals $p_i^{m_j}/p_j^{m_i}$, however such integrals are clearly reducible 
(they are ratios of Killing tensors) and hence do not add to integrability.

 \bigskip

{\bf Acknowledgment.}
This work was completed during several research trips of the author.
I am grateful for hospitality of IMPA at Rio de Janeiro and UNESP at campus S\~ao Jos\'e do Rio Preto in Brasil,
IHES in France, Isaak Newton Institute of Cambridge and the University of Loughborough in UK.

The research leading to these results was partially supported by the Tromsø Research Foundation 
(project “Pure Mathematics in Norway”) and the UiT Aurora project MASCOT.


\begin{thebibliography}{50}

\bibitem{AA}
S.\ Agafonov, T.\ Alves, {\it Fractional-linear integrals of geodesic flows on surfaces and Nakai's geodesic 4-webs},
Advances in Geometry {\bf 24}, no. 2, 263--273 (2024).

\bibitem{AS}
S.\ Agapov, V.\ Shubin, {\it Rational integrals of 2-dimensional geodesic flows: new examples}, 
J.\ Geom.\ Phys.\ {\bf 170}, 104389 (2021).

\bibitem{AHT}
A.\ Aoki, T.\ Houri, K.\ Tomoda, {\it Rational first integrals of geodesic equations and generalised hidden symmetries},
Class.\ Quantum Grav.\ {\bf 33}, 195003 (2016).

\bibitem{Coe}
S.\ Coen, {\it Factoriality of a ring of holomorphic functions}, Compositio Mathematica {\bf 29}, no. 2, 191-196 (1974).

\bibitem{Co}
C.\,D.\ Collinson, {\it A note on the integrability conditions for the existence of rational first integrals
of the geodesic equations in a Riemannian space}, Gen.\ Rel.\ Grav.\ {\bf 18}, 207 (1986).

\bibitem{Da}
G.\ Darboux, {\it M\'emoire sur les \'equations diff\'erentielles alg\'ebriques du
premier ordre et du premier degr\'e}, Bull.\ Sc.\ Math.\ 2\`eme s\'erie {\bf 2}, 60--96, 123--144, 151--200 (1878);
{\it Lecons sur la theorie generale des surfaces III}, Chelsea Publishing (1896).

\bibitem{D}
R.P.\ Delong, Jr., {\it Killing tensors and the Hamilton–Jacobi equation}, PhD thesis, University of Minnesota (1982).

\bibitem{E}
M.\ Eastwood, {\it Notes on projective differential geometry}, in: Symmetries and Overdetermined
Systems of Partial Differential Equations, IMA Volumes {\bf 144}, 41--60, Springer (2007).

\bibitem{EM}
M.\ Eastwood, V.\ Matveev, {\it Metric connections in projective differential geometry}, in: Symmetries and Overdetermined
Systems of Partial Differential Equations, IMA Volumes {\bf 144}, 339--350, Springer (2007).

\bibitem{GR}
R.\ Gunning, H.\ Rossi, {\it Analytic Functions of Several Complex Variables}, AMS Chelsea Publ.\ Series {\bf 368} (2009).


\bibitem{KoZ}
V.\ Kozlov, {\it Symmetries, Topology and Resonances in Hamiltonian Mechanics},
Ergeb.\ Math.\ Grenzgeb.\ (3) {\bf 31}, Springer, Berlin (1996).

\bibitem{Ko}
V.\ Kozlov, {\it On rational integrals of geodesic flows}, Regul.\ Chaotic Dyn.\ {\bf 19}, 601--606 (2014). 

\bibitem{Kr1}
B.\ Kruglikov, {\it Invariant characterization of Liouville metrics and polynomial integrals}, 
J.\ Geom.\ Phys.\ {\bf 58}, 979--995  (2008). 

\bibitem{Kr2}
B.\ Kruglikov, {\it On fractional-linear integrals of geodesics on surfaces},
Russ.\ J.\ Math.\ Phys.\ {\bf 32}, no.1, 97--104 (2025).

\bibitem{KL1}
B.\ Kruglikov, V.\ Lychagin, {\it Geometry of Differential equations},
Handbook on Global Analysis, D.Krupka and D.Saunders Eds., {\bf 1214}, 725-771, Elsevier Sci. (2008)

\bibitem{KL2}
B.\ Kruglikov, V.\ Lychagin, {\it Compatibility, multi-brackets and integrability of systems of PDEs},
Acta Appl.\ Math.\ {\bf 109}, 151--196 (2010).

\bibitem{KM}
B.\ Kruglikov, V.\ Matveev, {\it The geodesic flow of a generic metric does not admit nontrivial integrals polynomial
in momenta}, Nonlinearity {\bf 29}, 1755--1768 (2016).

\bibitem{MP}
A.\,J.\ Maciejewski, M.\ Przybylska, {\it Darboux polynomials and first integrals of natural polynomial
Hamiltonian systems}, Phys.\ Lett.\ A {\bf 326}, 219 (2004).

\bibitem{Te}
V.\ Ten, {\it Local integrals of geodesic flows}, Regul.\ Chaotic Dyn.\ {\bf 2}, 87--89 (1997). 

\bibitem{Th}
G.\ Thompson, {\it Killing tensors in spaces of constant curvature}, Jour.\ Math.\ Phys.\ {\bf 27}, 2693--2699 (1986).

\bibitem{Wh}
E.T.\ Whittaker, {\it A Treatise on the Analytical Dynamics of Particles and Rigid Bodies},
Cambridge University Press (1937).

\bibitem{Wo}
T.\ Wolf, {\it Structural equations for killing tensors of arbitrary rank}, 
Computer Physics Comm.\ {\bf 115}, no.\ 2--3, 316--329 (1998). 

 \end{thebibliography}
\end{document}